\documentclass[11pt]{amsart}
\usepackage{amsmath, amsthm, amssymb}
\usepackage[all]{xy}
\usepackage{mathtools}
\usepackage{enumitem}
\usepackage{mathrsfs}
\usepackage{tikz}

\newcommand{\G}{\mathcal{G}}
\newcommand{\K}{{H}}
\newcommand{\Gt}{\mathcal{G}_{\text{tight}}}
\newcommand{\mcH}{\mathbin{\mathcal{H}}}
\newcommand{\mcmu}{\mathbin{\mu}}

\newcommand{\ko}{{\K^{\scriptscriptstyle{(0)}}}}
\newcommand{\goo}{{G^{\scriptscriptstyle{(0)}}}}
\newcommand{\id}{{\rm{id}}}
\newcommand{\iso}{\operatorname{Iso}}
\newcommand{\fE}{\widehat{E}_0}
\newcommand{\tfE}{\widehat{E}_{\text{tight}}}
\newcommand{\up}{\uparrow\!}
\newcommand{\down}{\downarrow\!}

\DeclareMathOperator{\Ker}{Ker}

\newtheorem{prop}{Proposition}[section]
\newtheorem{thm}[prop]{Theorem}
\newtheorem{cor}[prop]{Corollary}
\newtheorem{lem}[prop]{Lemma}
\theoremstyle{definition}

\newtheorem{exmp}[prop]{Example}
\newtheorem{rem}[prop]{Remark}

\newlist{thmenum}{enumerate}{10}
\setlist[thmenum,1]{label=\textnormal{(\alph*)}}
\setlist[thmenum,2]{label=\textnormal{(\roman*)}}

\begin{document}

\title[Amenability and Uniqueness]{Amenability and Uniqueness for Groupoids Associated with Inverse Semigroups}
\author{Scott M. LaLonde and David Milan}
\address{Department of Mathematics\\
The University of Texas at Tyler\\
3900 University Boulevard\\
Tyler, TX 75799}
\email{slalonde@uttyler.edu, dmilan@uttyler.edu}

\date{\today}

\begin{abstract}
	We investigate recent uniqueness theorems for reduced $C^*$-algebras of Hausdorff \'{e}tale groupoids in the context of inverse semigroups. In many 
	cases the distinguished subalgebra is closely related to the structure of the inverse semigroup. In order to apply our results to full $C^*$-algebras, we also 
	investigate amenability. More specifically, we obtain conditions that guarantee amenability of the universal groupoid for certain classes of inverse semigroups. 
	These conditions also imply the existence of a conditional expectation onto a canonical subalgebra.
\end{abstract}

\maketitle


\section{Introduction} 
\label{sec:intro}

Paterson \cite{PatersonBook} defined the universal groupoid $\G(S)$ of an inverse semigroup $S$ and showed that the full and reduced $C^*$-algebras of 
the universal groupoid are isomorphic to the full and reduced $C^*$-algebras of the inverse semigroup, respectively. Exel \cite{ExelBig} gave a streamlined 
account of this construction and the more general construction of the groupoid of germs of an inverse semigroup action. Crucially, he defined an action of 
$S$ on the closure of the space of ultrafilters of idempotents, called the tight spectrum of $S$, with corresponding groupoid of germs called the tight groupoid 
$\Gt(S)$. Often an analog of the Cuntz-Krieger relations holds for the $C^*$-algebra of the tight groupoid, as has been shown for many classes of inverse 
semigroups: directed graphs \cite{PatersonGraph}, semigroupoids \cite{ExelBig}, Kellendonk's tiling semigroups \cite{ExelGonStar}, and Spielberg's path 
categories \cite{DonsigMilan}, to name some. 

In \cite{BNRSW} a very general version of the Cuntz-Krieger Uniqueness Theorem is proved for the reduced $C^*$-algebra of a Hausdorff \'{e}tale groupoid 
$\G$. It is shown that the reduced $C^*$-algebra of the interior of the isotropy of $\G$ embeds as a subalgebra $M_r$ of $C_{r}^*(\G)$, and that the pure 
states of $M_r$ with unique extension to $C_{r}^*(\G)$ are dense in the set of pure states on $M_r$. As a consequence, a $*$-homomorphism with domain 
$C_{r}^*(\G)$ is injective if and only if its restriction to $M_r$ is injective.

In this paper we are interested in understanding the uniqueness theorem of \cite{BNRSW} in the context of the groupoid of germs of an inverse semigroup 
action. We show that the universal groupoid of the centralizer of the idempotent semilattice $E(S)$ is contained in the interior of the isotropy bundle of the 
universal groupoid. In the case that $S$ is cryptic (that is, the $\mcH$-relation on $S$ is a congruence), the two groupoids are equal. 

More generally, the kernel of any action $\alpha$ of an inverse semigroup on a locally compact Hausdorff space generates a subgroupoid of the groupoid of 
germs $\G_{\alpha}$ of the action. We show that this subgroupoid is contained in the interior of the isotropy bundle of $\G_{\alpha}$ and give conditions under 
which equality holds. In the case of the standard action of $S$ on its tight spectrum, we have equality provided the semilattice of idempotents of $S$ is 
$0$-disjunctive. Thus there is a uniqueness theorem for the tight $C^*$-algebras of many inverse semigroups, where the distinguished subalgebra is generated 
by the centralizer of $E(S)$. This is a fairly satisfactory result that connects the algebraic structure of the inverse semigroup to the structure of its $C^*$-algebra 
for a large class of inverse semigroups. 

It is well known that every inverse semigroup $S$ can be described as an extension of the centralizer of $E(S)$ by a fundamental inverse semigroup. In light of
the results mentioned above, there should be a useful connection between the structure of the universal groupoid of an inverse semigroup and this extension. With 
that in mind, we conclude the paper with an investigation of the amenability of the universal groupoid. In particular, we obtain conditions under which the 
aforementioned extension gives us a ``two out of three'' theorem regarding amenability. Along the way, we will see that these conditions guarantee the existence
of a conditional expectation in many cases.
 
The structure of the paper is as follows. In Section \ref{sec:prelim}, we present the necessary background material on inverse semigroups and groupoids. In 
preparation for our work with the groupoid of the centralizer of $E(S)$, we collect some results on Clifford semigroups and group bundles in Section 
\ref{sec:clifford}. We present our first uniqueness theorem in Section \ref{sec:unique}. Section \ref{sec:actions} deals with groupoids of inverse semigroup actions; 
we present generalized uniqueness results there. Finally, our results on amenability come in Section \ref{sec:amenability}.


\section{Preliminaries}
\label{sec:prelim}

An \textit{inverse semigroup} is a semigroup $S$ such that for each $s$ in $S$, there exists a unique $s^*$ in $S$ such that 
\[
	s = s s^* s\quad \text{and}\quad s^* = s^* s s^*.
\]
We quickly outline a number of important facts about inverse semigroups. A thorough treatment of the subject can be found in~\cite{LawsonBook}.

There is a natural partial order on $S$ defined by $s \leq t$ if $s = te$ for some idempotent $e$. Given $s$ in $S$ we let $\down s = \{ t \in S : t \leq s\}$ and $\up s = \{ t \in S : t \geq s\}$. The subsemigroup $E = E(S)$ of idempotents of $S$ is commutative, and hence forms a (meet) semilattice for the natural partial order with $e \wedge f= ef$ for $e,f$ in $E(S)$.

There are a number of important equivalence relations defined on an inverse semigroup. The two that play an important role here are the $\mcH$- and 
$\mu$-relations. The $\mcH$-relation is defined by $s \mcH t$ if and only if $s^* s = t^* t$ and $s s^* = t t^*$. The $\mcH$-class of an idempotent $e$, denoted 
$H_e$, is the maximum subgroup of $S$ with identity $e$. The $\mu$-relation is defined by $s \mcmu t$ if and only if $ses^* = tet^*$ 
for all $e \in E$. We denote the $\mu$-class of an idempotent $e$ by $Z_e$; it is also a group with identity $e$, hence a subgroup of $H_e$. In fact, it is 
well-known that $\mu \subseteq \mcH$. If $\mu = \mcH$, we say that $S$ is \emph{cryptic}.

An equivalence relation $\rho$ on $S$ is a {\it congruence} if $(a,b), (c,d) \in \rho$ implies $(ac, bd) \in \rho$. The $\mu$-relation is always a congruence, though
$\mcH$ need not be in general. A relation $\rho$ is \emph{idempotent separating} if each $\rho$-class contains at most one idempotent. Both $\mcH$ and $\mu$
are idempotent separating. Moreover, it is well known that $\mu$ is the maximal idempotent separating congruence. Notice then that $S$ is cryptic if and only if
$\mcH$ is a congruence. It is also worth noting that the quotient of an inverse semigroup by a congruence is again an inverse semigroup.

Given a relation $\rho$ on $S$, we let $\Ker \rho$ denote the union of $\rho$-classes of idempotents. We refer to $\Ker \rho$ as the {\it kernel of $\rho$} 
(despite the typographical distinction in \cite[Section 5.1]{LawsonBook}). If $\rho$ is a congruence, it is straightforward to see that $\Ker \rho$ is a 
\emph{normal} subsemigroup---it contains all the idempotents and satisfies $s^* (\Ker \rho ) s \subseteq \Ker \rho$ for all $s \in S$. One can quickly see that 
\[ 
	\Ker \rho = \{ st^* \in S : \rho(s) = \rho(t) \}.
\]
Observe that the kernel of the $\mu$-relation is
\[
	\bigcup_{e \in E} Z_e = \{ s \in S : se = es \text{ for all } s \in S\} = Z(E),
\]
which we refer to as the \emph{centralizer of the idempotents}. This subsemigroup has a particularly simple form---it is just a union of
groups. Such an inverse semigroup is said to be \emph{Clifford}. Equivalently, an inverse semigroup $S$ is Clifford if and only $s^* s = s s^*$ for all $s \in S$. 

Notice that in a Clifford semigroup, every element is $\mu$-related to an idempotent. At the other end of the spectrum are the inverse semigroups for which the 
$\mu$-class of any idempotent is trivial. We say that $S$ is {\it fundamental} if $\mu$ is the identity relation. Notice that the inverse semigroup $S/\mu$ is always 
fundamental, and its semilattice of idempotents is isomorphic to $E$. Thus every inverse semigroup is an extension of a Clifford semigroup by a fundamental 
inverse semigroup.   

There is a useful technique for constructing examples of fundamental inverse semigroups with a specified semilattice of idempotents. Given a semilattice $E$, we let 
$T_E$ denote the isomorphisms of principal order ideals in $E$. Then $T_E$ is a fundamental inverse semigroup, called the {\it Munn semigroup} 
of the semilattice $E$. Munn \cite{MunnFun} proved that every fundamental inverse semigroup with semilattice $E$ is isomorphic to a full inverse subsemigroup of $T_E$.


Much of what we aim to do in this paper relies on groupoid techniques, so we present the necessary background here. Recall that a \emph{groupoid} consists 
of a set $G$ together with a set $G^{\scriptscriptstyle{(2)}} \subseteq G \times G$ (called \emph{composable pairs}), a map $(\gamma, \eta) \mapsto \gamma \eta$ from 
$G^{\scriptscriptstyle{(2)}} \to G$, and an involution $\gamma \mapsto \gamma^{-1}$ from $G \to G$ satisfying:
\begin{enumerate}
	\item $\gamma (\eta \xi) = (\gamma \eta) \xi$ whenever $(\gamma, \eta), (\eta, \xi) \in G^{\scriptscriptstyle{(2)}}$;
	\item $\gamma^{-1} (\gamma \eta) = \eta$ and $(\gamma \eta) \eta^{-1} = \gamma$ for $(\gamma, \eta) \in \G^{\scriptscriptstyle{(2)}}$.
\end{enumerate}
Elements $u \in G$ satisfying $u = u^2 = u^{-1}$ are called \emph{units}, and the set of all units is denoted by $G^{\scriptscriptstyle{(0)}}$ and called the 
\emph{unit space} of $G$. There are always two continuous surjections $r, d : G \to G^{\scriptscriptstyle{(0)}}$, defined by
\[
	r(\gamma) = \gamma \gamma^{-1}, \quad d(\gamma) = \gamma^{-1} \gamma
\]
and called the \emph{range} and \emph{source}, respectively. Among other things, $r$ and $d$ can be used to characterize composability:
$(\gamma, \eta) \in G^{\scriptscriptstyle{(2)}}$ if and only if $d(\gamma) = r(\eta)$. These maps can also be used to fiber $G$: given $u \in G^{\scriptscriptstyle{(0)}}$, 
it is fairly standard to write
\[
	G^u = r^{-1}(u), \quad G_u = d^{-1}(u), \quad G_u^u = G^u \cap G_u.
\]
It is easily checked that $G^u_u$ is a group with identity $u$, called the \emph{isotropy} (or \emph{stabilizer}) group at $u$. If we have $G^u = G_u = G_u^u$ for 
all $u$ (equivalently, $r(\gamma) = d(\gamma)$ for all $\gamma \in G$), we say $G$ is a \emph{group bundle}. Every groupoid contains a maximal group bundle, 
namely $\iso(G) = \bigcup_{u \in \goo} G_u^u$, which is called the \emph{isotropy bundle}. Since $r$ and $d$ are continuous, $\iso(G)$ is necessarily closed in $G$.

We are primarily interested in the case where $G$ is equipped with a topology such that multiplication and inversion are continuous. In particular, we will only 
deal with \emph{\'{e}tale groupoids}, where the topology is locally compact and the range and source maps are local homeomorphisms. In general, $G$ need
not be Hausdorff, though we will often need to add that requirement to our results. 

We will consider two closely related properties that frequently appear in the groupoid literature. An \'{e}tale groupoid $G$ is said to be \emph{essentially principal} if 
the interior of $\iso(G)$ is equal to $\goo$. An \'{e}tale groupoid $G$ is \emph{effective} if any open set $U \subseteq G \backslash \goo$ contains an element 
$\gamma$ with $d(\gamma) \neq r(\gamma)$. It is known (see \cite[Lemma 3.1]{BCFS}, for example) that under mild hypotheses these two properties are the same. 

\begin{prop}
	If $G$ is essentially principal, then it is effective. If $G$ is assumed to be Hausdorff, then the converse holds.
\end{prop}
\begin{proof}
	Assume first that $G$ is essentially principal, and suppose there is an open set $U \subseteq G \backslash \goo$ with the property that $d(\gamma) = 
	r(\gamma)$ for all $\gamma \in U$. Then $U \subseteq \iso(G)$, and $U \cap \goo = \emptyset$, so the interior of $\iso(G)$ is not equal to $\goo$. This 
	contradicts the fact that $G$ is essentially principal.
	
	Now assume that $G$ is Hausdorff. Then $\goo$ is clopen in $G$. If $G$ is not essentially principal, then $U = \iso(G)^\circ \backslash \goo$ is a 
	nonempty open subset of $G \backslash \goo$, and $d(\gamma) = r(\gamma)$ for all $\gamma \in U$. Thus $G$ is not effective.
\end{proof}

As described in Paterson \cite{PatersonBook} and Exel \cite{ExelBig}, inverse semigroups give rise to \'{e}tale groupoids in a very natural way. Let $S$ be an inverse semigroup, and let $X$ be a locally compact Hausdorff space. The \emph{symmetric inverse monoid} on $X$ is the set of all ``partial one-to-one maps'' on $X$,
\[
	\mathcal{I}(X) = \{ f : U \to V \mid U, V \subseteq X \text{ and } f \text{ is bijective} \},
\]
which is an inverse semigroup with multiplication given by composition on the largest possible domain. 
An \emph{action} of $S$ on $X$ is a semigroup homomorphism $\alpha : S \to \mathcal{I}(X)$ 
such that each $\alpha_s$ is continuous with open domain $D^{\alpha}_{s^*s}$, and the domains of the $\alpha_s$ cover $X$. Given such an action, we can build 
a groupoid $\G_\alpha$ as follows. Define an equivalence relation on $S * X = \{(s, x) \mid s \in S, x \in D^\alpha_{s^*s}\}$ by $(s, x) \sim (t, y)$ if and only if $x = y$ 
and there is an idempotent $e \in E(S)$ with $x \in D^\alpha_e$ and $se = te$. Now set $\G_\alpha = S*X/\sim$, and define
\[
	[t, \alpha_s(x)][s, x] = [ts, x], \quad [s, x]^{-1} = [s^*,\alpha_s(x)].
\]
These operations make $\G_\alpha$ into a groupoid with range and source 
\[
	r([s,x]) = [ss^*, \alpha_s(x)], \quad d([s, x]) = [s^*s, x],
\]
which allows us to identify $\G_\alpha^{\scriptscriptstyle{(0)}}$ with $X$. Moreover, $\G_\alpha$ naturally inherits a topology from that of $X$: given $s \in S$ and 
an open set $U \subseteq D^\alpha_{s^*s}$, define
\[
	\Theta(s, U) = \{ [s, x] \mid x \in U\}.
\]
These sets form a base for a topology on $\G_\alpha$, which is second countable if $X$ is. Furthermore, this topology makes $\G_\alpha$ into an \'{e}tale 
groupoid, which need not be Hausdorff in general. The groupoid $\G_\alpha$ is referred to as the {\it groupoid of germs} of the action $\alpha$.

There are two canonical actions, and hence groupoids, associated to any inverse semigroup $S$. Recall that a {\it filter} in $E=E(S)$ is a nonempty subset 
$F \subseteq E$ that is closed under multiplication, closed upward in the partial order, and does not contain the zero element of $S$ (if one exists). We denote the 
set of filters by $\fE$. Filters are in one-to-one correspondence with the set of characters on $E$ and thus they can be topologized using the product topology. A 
useful characterization of this topology on $\fE$ is that it has a base consisting of open sets of the following form:
\[ 
	N^{e}_{f_1,f_2, \dots, f_n} := \{ F \in \fE : e\in F \text{ and } f_i \not \in F \text{ for } 1 \leq i \leq n \}. 
\] 
There is a standard action $\beta$ of $S$ on $\fE$ defined as follows. Given $s \in S$, let $D^{\beta}_{s^* s} = \{ F \in \fE : s^* s \in F\}$. Then 
$\beta_s : D^{\beta}_{s^* s} \to D^{\beta}_{s s^*}$ is a continuous bijection defined by
\[
	\beta_s(F) = \{ f \in E : f \geq s e s^*, \text{ for some } e \in F \}.
\] 
for $F \in D^{\beta}_{s^* s}$. The {\it universal groupoid} of $S$, denoted by $\G(S)$, is the groupoid of germs of this action. It is well-known that the full and reduced
$C^*$-algebras of $\G(S)$ are isomorphic to $C^*(S)$ and $C_r^*(S)$, respectively (see Theorems 4.4.1 and 4.4.2 of \cite{PatersonBook}).

The other natural action of $S$ involves the set of \emph{tight filters}, denoted $\tfE$, which is defined to be the closure of the ultrafilters in $\fE$. This space is 
invariant under the action $\beta$. Following Exel, we use $\theta$ to denote the restriction of $\beta$ to the space of tight filters. The groupoid of germs of $\theta$ 
is called the {\it tight groupoid} of $S$ and is often denoted by $\Gt(S)$. 

One example that demonstrates the importance of the tight groupoid comes from the inverse semigroup of a directed graph $\Lambda$, discussed in detail in 
\cite{PatersonGraph}. Paterson shows that if $\Lambda$ is row-finite and does not contain sinks, then the tight groupoid is isomorphic to the usual path groupoid 
of $\Lambda$ (see the discussion following \cite[Theorem 1]{PatersonGraph}). Even without the assumption that $\Lambda$ is row-finite, the $C^*$-algebra of the 
tight groupoid is isomorphic to the graph algebra of $\Lambda$ \cite[Corollary 1]{PatersonGraph}.

The universal and tight groupoids need not be Hausdorff, though we will need them to be so for some results. We can guarantee that they are both Hausdorff
by requiring that $S$ is \emph{Hausdorff}: for all $s, t \in S$, the set ${\down s} \cap {\down t}$ is finitely generated as a lower set (cf. \cite{SteinbergSimple}).
It is well-known that any $E$-unitary (or $0$-$E$-unitary) inverse semigroup is automatically Hausdorff. More generally, the groupoid of germs of an action
$\alpha : S \to \mathcal{I}(X)$ is Hausdorff if $S$ is Hausdorff and the sets $D_e^\alpha$ are all clopen in $X$ by \cite[Proposition 2.3]{SteinbergSimple}.


\section{Clifford Semigroups and the Groupoid of the Centralizer}
\label{sec:clifford}
One of the main goals of this paper is to understand the relationship between the Clifford semigroup $Z(E)$ and the group bundle $\iso(\G(S))^\circ$ for 
an inverse semigroup $S$. Therefore, we need to know something about the structure of the universal groupoid associated to a Clifford semigroup. 

Suppose $S$ is a Clifford semigroup. We aim to show that the groupoid $\G(S)$ is a group bundle over $\fE$, and that the fibers of $\G(S)$ are closely 
related to the maximal subgroups of $S$. In particular, the fibers of $\G(S)$ are isomorphic to direct limits of subgroups of $S$, and they are equal to the maximal
subgroups of $S$ at principal filters.

Let $F \in \fE$ be a filter, and consider the family $\{H_e\}_{e \in F}$. The filter $F$ is, by definition, \emph{downward directed}: if $e, e' \in F$, then there 
is an $f \in F$ such that $f \leq e$ and $f \leq e'$. Moreover, there is a natural directed system of morphisms for the family $\{H_e\}_{e \in F}$: if $e \leq f$, 
define a restriction map $\varphi^f_e : H_f \to H_e$ by
\[
	\varphi^f_e(s) = se.
\]
It is easy to see that $\varphi^f_e$ is a homomorphism, thanks to the fact that idempotents are central:
\[
	\varphi^f_e(st) = (st)e = ste^2 = (se)(te) = \varphi^f_e(s) \varphi^f_e(t).
\]
Moreover, we have $\varphi_e^e = \id$ for each $e \in E(S)$, and if $e \leq f \leq g$, then
\[
	\varphi^g_f \circ \varphi^f_e(s) = \varphi^g_f(se) = sef = se = \varphi^g_e(s).
\]
Thus we can build the direct limit $\varinjlim_F H_e$, which we denote by $H_F$. We let $\varphi_e : H_e \to H_F$ denote the standard maps into the direct limit,
and for $s \in H_e$ we write $[s] = \varphi_e(s)$ to represent the equivalence class of $s$ in $H_F$.  Observe that if $s \in H_e$ and $t \in H_f$, then $[s] = [t]$ in $H_F$ if 
and only if there is an idempotent $g \in F$ such that $\varphi^e_g(s) = \varphi^f_g(t)$, or $sg = tg$. In other words, $[s]=[t]$ if and only if $s$ and $t$ have a common 
restriction.

We are now prepared to prove the main structure theorem for $\G(S)$ when $S$ is a Clifford semigroup. We should point out that the following result has already 
been observed by Paterson \cite[Section 4.3]{PatersonBook}. However, we find it instructive to include the details here.

\begin{thm}
\label{thm:clifford}
	Let $S$ be a Clifford semigroup, and let $\G(S)$ denote its universal groupoid. Then:
	\begin{enumerate}
		\item $\G(S)$ is a group bundle.
		\item Given $F \in \fE$, the fiber $\G(S)_F$ can be identified with $H_F = \varinjlim_F H_e$.
		\item If $F = {\uparrow e}$ is a principal filter, then $\G(S)_F \cong H_e$.
	\end{enumerate}
\end{thm}
\begin{proof}
	To prove (1), it suffices to show that $d([s, F]) = r([s, F])$ for all $[s, F] \in \G(S)$. Recall that $d([s, F]) = [s^*s, F]$ and $r([s, F]) = [ss^*, \beta_s(F)]$, 
	where $\beta$ denotes the standard action of $S$ on $\fE$. We have $s^*s = ss^*$ by assumption, so we just need to check that $\beta_s(F) = F$. 
	Recall that
	\[
		\beta_s(F) = \{ f \in E(S) : ses^* \leq f \text{ for some } e \in F \}.
	\]
	If $e \in F$, then clearly $ses^* \in \beta_s(F)$. But idempotents in $S$ are central, so
	\[
		(ses^*)e = sees^* = ses^*.
	\]
	Thus $ses^* \leq e$, which forces $e \in F$ as well. Therefore, $F \subseteq \beta_s(F)$. On the other hand, suppose $f \in \beta_s(F)$. Then there 
	is an $e \in F$ with $ses^* \leq f$. But $ses^* = ss^*e = s^*se$, so $s^*se \leq f$. Note that $s^*s \in F$ and $e \in F$, so $s^*se \in F$, and therefore 
	$f \in F$ since $F$ is a filter. Thus $\beta_s(F) \subseteq F$ as well.
	
	For (2), we first need to show that there is a compatible family of maps $\psi_e : H_e \to \G(S)_F$. Well, define $\psi_e : H_e \to \G(S)_F$ by
	$\psi_e(s) = [s, F]$. It is immediate that $\psi_e$ is a homomorphism. Moreover, if $e \leq f$, we have
	\[
		\psi_e \circ \varphi^f_e(s) = \psi_e(se) = [se, F] = [s, F] = \psi_f(s).
	\]
	Thus the $\{\psi_e\}$ form a compatible family of morphisms, so there is a unique homomorphism $\Phi : H_F \to \G(S)_F$. A quick calculation shows 
	that $\Phi$ must be given by $\Phi([s]) = [s, F]$. It is clear that $\Phi$ is surjective, and if $\Phi([s]) = \Phi([t])$, then we have $[s, F] = [t, F]$, so there 
	exists $e \in E(S)$ such that $se = te$. But then $[s] = [t]$ in $H_F$, so $\Phi$ is injective. Therefore, $\G(S)_F$ is naturally isomorphic to $H_F$.

	Now suppose $F = {\uparrow e}$ is a principal filter. With (2) in hand, we only need to show that $\varinjlim_F H_f \cong H_e$ to establish (3). Note first 
	that we have a compatible family of morphisms $\psi_f : H_f \to H_e$, namely $\psi_f = \psi^f_e$. Thus the universal property of the direct limit guarantees 
	that there is a unique homomorphism $\Psi : H_F \to H_e$, and we just need to verify that $\Psi$ is an isomorphism. First, we must have 
	$\Psi \circ \varphi_e = \varphi_e^e = \id$. On the other hand, it is easy to verify that the diagram
	\[
	\xymatrix{
		H_f \ar[rr]^{\varphi^f_g} \ar[dr]_{\varphi_f} \ar@(d,l)[ddr]_{\varphi_f} & & H_g \ar[dl]^{\varphi_g} \ar@(d,r)[ddl]^{\varphi_g} \\
			& H_F \ar[d]^{\varphi_e \circ \Psi} & \\
			& H_F &
	}
	\]
	commutes. The identity function $\id : H_F \to H_F$ also makes the diagram commute, so we must have $\varphi_e \circ \Psi = \id$. Thus $\Psi$ is an 
	isomorphism. It follows that $\G(S)_F \cong H_e$, and it is worth noting that this isomorphism is nothing more than $\Phi(s) = [s, F]$.
\end{proof}

Given an inverse semigroup $S$, we want to use Theorem \ref{thm:clifford} to analyze the structure of $\G(Z)$, where $Z = Z(E)$ is the centralizer of the 
idempotents of $S$. In particular, it is natural to consider the relationship between the universal groupoid of $Z$ and the isotropy bundle of $\G(S)$.

\begin{prop} 
\label{prop:Z} 
	Let $S$ be an inverse semigroup, and let $Z = Z(E)$ denote the centralizer of the idempotents of $S$. Then there is a natural 
	embedding $\iota : \G(Z) \hookrightarrow \G(S)$ given by 
	\[
		\iota([s, F]) = [s, F].
	\]
	Moreover, $\iota(\G(Z))$ is open in $\G(S)$, so $\G(Z)$ can be identified with a subset of $\iso(\G(S))^\circ$.
\end{prop}
\begin{proof}
	Clearly $\iota$ is a well-defined, injective groupoid homomorphism. We just need to verify that it is a homeomorphism onto its range, and that its 
	range is open in $\G(S)$. Recall that a base for the topology on $\G(Z)$ is given by sets of the form
	\[
		\Theta_Z(s, U) = \bigl\{ [s, F] \in \G(Z) \mid F \in U \bigr\},
	\]
	where $s \in Z$ and $U \subseteq \fE$ is open. It is immediate that 
	\[
		\iota(\Theta_Z(s, U)) = \bigl\{ [s, F] \in \G \mid F \in U \bigr\} = \Theta(s, U) ,
	\]
	so $\iota$ maps basic open sets in $\G(Z)$ to basic open sets in $\G(S)$. It follows that $\iota$ is a homeomorphism onto its range. Moreover, this shows 
	that if $[s, F] \in \iota(\G(Z))$ and $U \subseteq \fE$ is any open set, then $\Theta(s, U) \subseteq \iota(\G(Z))$. Thus $\iota(\G(Z))$ is open in $\G(S)$. 
	Theorem \ref{thm:clifford} guarantees that $\G(Z)$ is a group bundle, hence $\iota(\G(Z)) \subseteq \iso(\G(S))^\circ$.
\end{proof}

We will henceforth identify $\G(Z)$ with its image in $\G(S)$. We have just shown that $\G(Z) \subseteq \iso(\G(S))^\circ$, and Theorem \ref{thm:clifford} also 
tells us about the structure of the fibers---at any principal filter $\up e \in \fE$, we have $\G(Z)_{\up e} = Z_e$, where $Z_e$ denotes the $\mu$-class of $e$. It 
is then natural to ask whether this group equals the isotropy group $\G_{\up e}^{\up e}$. In general, the answer is no.

\begin{prop}
\label{prop:hclass}
	Let $S$ be an inverse semigroup. For any $e \in E(S)$, the isotropy group $\G(S)_{\up e}^{\up e}$ is equal to the $\mcH$-class $H_e$.
\end{prop}
\begin{proof}
	Fix $e \in E(S)$. We claim that the map $s \mapsto [s, \up e]$ defines an isomorphism of $H_e$ onto $\G(S)_{\up e}^{\up e}$. Given $s \in H_e$, we need to 
	first show that $[s, \up e] \in \iso(\G(S))$. That is, we must have $[s^*s, \up e] = [ss^*, \theta_s(\up e)]$. Well, $s^*s = ss^* = e$ since $s \in H_e$, so we just
	need to check that $\theta_s(\up e) = \up e$. But it is easy to see that
	\[
		\theta_s(\up e) = \theta_s(\up s^*s) = \up ss^* = \up s^*s = \up e.
	\]
	Thus $[s, \up e] \in \G(S)_{\up e}^{\up e}$.
	
	Now suppose that $[s, \up e] \in \G(S)_{\up e}^{\up e}$, so $[s^*s, \up e] = [ss^*, \theta_s(\up e)]$. Then $\up e = \theta_s(\up e) = \up ses^*$, so we must 
	have $e = ses^*$. Put $t = se$. Then
	\[
		t^*t = (se)^*(se) = es^*se = e,
	\]
	since $e \leq s^*s$. Similarly,
	\[
		tt^* = (se)(se)^* = sees^* = ses^* = e.
	\]
	Thus $t \in H_e$ and $[t, \up e] = [s, \up e]$, so $H_e$ is mapped onto the isotropy group $\G(S)_{\up e}^{\up e}$. It only remains to show that this mapping is
	injective. Suppose $s, t \in H_e$ and $[s, \up e] = [t, \up e]$. Then there is an idempotent $f \in \up e$ such that $sf = tf$. But $e \leq f$ and $se = s$, so we 
	have
	\[
		sf = sef = se = s.
	\]
	Similarly, $tf = t$, so $s = t$.
\end{proof}

In light of Propositions \ref{prop:Z} and \ref{prop:hclass}, we have the following chain of containments for any $e \in E(S)$:
\begin{equation}
\label{eq:chain}
	Z_e = \G(Z)_{\up e} \subseteq \iso(\G(S))^\circ_{\up e} \subseteq \iso(\G(S))_{\up e} = H_e.
\end{equation}
Clearly, the second containment may turn out to be proper, say if $\iso(\G(S))$ is not open in $\G(S)$. A more interesting observation is that the first containment may 
also be proper. That is, $\G(Z)$ may not be all of $\iso(\G(S))^{\circ}$ in general, as the next example shows. 

\begin{exmp}
\label{exmp:diamond}
Consider the semilattice $E=\{0,a,b,1\}$, where $0 \leq a,b \leq 1$, and $a$ is not comparable to $b$. The set $T_E$ of order isomorphisms of principal order ideals 
of $E$ is a fundamental inverse semigroup. The semilattice $E$ and the $\mathcal{D}$-classes of $T_E$ are pictured below. The inverse semigroup $B_2$ is a combinatorial Brandt semigroup consisting of the maps $a \mapsto b$, $a \mapsto a$, $b \mapsto a$, $b \mapsto b$, and $0$.  
\begin{center}\begin{tikzpicture}
	\node (1) at (2,2) {$1$};
	\node (a) at (1,1) {$a$};
	\node (b) at (3,1) {$b$};
	\node (0) at (2,0) {$0$};
	\draw (0)--(a)--(1)--(b)--(0);
	
	\node (z) at (5,2) {$\mathbb{Z}_2 = \{1,-1\}$};
	\node (b) at (5,1) {$B_2 \setminus \{0\}$};
	\node (0) at (5,0) {$0$};
	\draw (0)--(b)--(z);

\end{tikzpicture}
\end{center}
Since $T_E$ is fundamental, $Z(E) = E$ and hence $\G(Z)$ is principal. However, there is a nontrivial isomorphism of $\down 1$ sending $a$ to $b$, which
we denote by $-1$. It is easy to verify that 
\[ 
	[-1,\{1\}] \in \iso(\G(S)) \text{ and } \Theta(-1, N^{1}_{a,b}) = \{[-1,\{1\}]\}.
\] 
Thus $[-1,\{1\}] \in \iso(\G(S))^{\circ}$. Since $[-1,\{1\}]$ is not a unit, we conclude that $\G(Z)$ is properly contained in $\iso(\G(S))^{\circ}$.
\end{exmp}

The real issue in Example \ref{exmp:diamond} is that the subgroup $Z_e$ at an idempotent $e$ may be properly contained in the maximal subgroup $H_e$. This 
happens precisely when the $\mu$-relation is properly contained in the $\mathcal{H}$-relation, or equivalently when $\mathcal{H}$ fails to be 
a congruence. Recall that an inverse semigroup $S$ is called \textit{cryptic} if $\mu=\mcH$. It is immediate that if $S$ is cryptic, then \eqref{eq:chain} is simply a 
chain of equalities, so we have $\G(Z)_{\up e} = \iso(\G(S))^\circ_{\up e}$ at any principal filter $\up e$. In fact, this condition is sufficient to guarantee equality at 
\emph{every} filter. To prove it, we will invoke the idea of a \emph{continuously varying group bundle}.

Let $G$ be a group bundle. Recall that $G$ is said to be \emph{continuously varying} if the map $u \mapsto G_u$ is continuous from $\goo \to \mathscr{C}(G)$, 
where $\mathscr{C}(G)$ is the set of all closed subsets of $G$ endowed with the Fell topology. (A detailed description of the Fell topology is given in Appendix H 
of \cite{TFB2}.) In particular, we say a groupoid has \emph{continuously varying stabilizers} if its isotropy bundle is continuously varying. It is well-known 
\cite{geoff,lalondethesis,renault91} that a group bundle $G$ varies continuously if and only if its range map (which is equal to the source map) is open, which in turn 
is equivalent to the existence of a Haar system on $G$. It is not necessary to assume that $G$ is Hausdorff.

We will add a fourth condition that is equivalent to a groupoid $G$ having continuously varying stabilizers, at least when $G$ is \'{e}tale. The following result is 
not necessary for what follows, and it is probably already known to groupoid specialists, but we include it since it gives us a nice characterization of when 
$\iso(G)$ is open.

\begin{prop}
	Let $G$ be a locally compact (not necessarily Hausdorff) \'{e}tale groupoid. The isotropy bundle $\iso(G)$ is open in $G$ if and only if $G$ has continuously
	varying stabilizers.
\end{prop}
\begin{proof}
	Suppose first that $\iso(G)$ is open. Then the restriction of the range map to $\iso(G)$ is an open map, which implies that $\iso(G)$ is continuously varying.
	
	Now suppose $\iso(G)$ is not open, i.e., $\iso(G)^\circ \neq \iso(G)$. Then there is a $u \in \goo$ such that $G_u^u \neq \iso(G)_u^\circ$.
	Now \cite[Lemma 3.3]{BNRSW} guarantees that there is a net $\{u_i\}$ in $\goo$ such that $u_i \to u$ and $G_{u_i}^{u_i} = \iso(G)_{u_i}^\circ$
	for all $i$. Since $\iso(G)^\circ$ is open in $G$, it is a continuously varying group bundle, so we have
	\[
		G_{u_i}^{u_i} = \iso(G)^\circ_{u_i} \to \iso(G)_u^\circ
	\]
	with respect to the Fell topology on $\mathscr{C}(\iso(G)^\circ)$. We need to check that this implies convergence in $\mathscr{C}(G)$.
	
	First observe that $G_u^u$ is closed in $G$, and $\iso(G)_u^\circ$ is closed in $G_u^u$ (since both are discrete groups), so $\iso(G)_u^\circ$
	is closed in $G$. Therefore, $\iso(G)_u^\circ \in \mathscr{C}(G)$. Now it is clear that conditions (a) and (b) of \cite[Lemma H.2]{TFB2} hold for the 
	net $\{G_{u_i}^{u_i}\}$, since $G_{u_i}^{u_i} \to \iso(G)_u^\circ$ in $\iso(G)^\circ$. Therefore, $G_{u_i}^{u_i} \to \iso(G)_u^\circ$ in $\mathscr{C}(G)$. 
	Since the Fell topology is Hausdorff and $G_u^u \neq \iso(G)_u^\circ$ by assumption, $G_{u_i}^{u_i}$ cannot converge to $G_u^u$. Hence $G$ does 
	not have continuously varying stabilizers.
\end{proof}

\begin{prop} 
\label{prop:cryptic}
	Let $S$ be an inverse semigroup with centralizer $Z$. If $S$ is cryptic, then $\G(Z) = \iso(\G(S))^{\circ}$.  
\end{prop}
\begin{proof}
	Since $\G(Z)$ and $\iso(\G(S))^\circ$ are open in $\G(S)$, they are both continuously varying. The set of principal filters is dense in $\fE$ by 
	\cite[Proposition 1.2]{PatersonClifford}, and we have $\G(Z)_{\up e} = \iso(\G(S))^\circ_{\up e}$ at every principal filter. Continuity then implies that 
	$\G(Z)_F = \iso(\G(S))^\circ_F$ at any filter $F \in \fE$.
\end{proof}


\section{Embeddings and Uniqueness}
\label{sec:unique}

We saw in the last section that if $S$ is an inverse semigroup with centralizer $Z$, then $\G(Z)$ always embeds into $\G(S)$ as an open subgroupoid. This gives
us an initial step toward stating a uniqueness theorem for $C_r^*(S)$ in terms of the distinguished subalgebra $C_r^*(Z)$. Of course this assumes that
the embedding $\G(Z) \hookrightarrow \G(S)$ yields an embedding at the level of $C^*$-algebras.

In \cite[Theorem 3.1(a)]{BNRSW}, it is proven that for a locally compact Hausdorff \'{e}tale groupoid $G$, there is a natural embedding $C_r^*(\iso(G)^\circ) 
\hookrightarrow C_r^*(G)$. Furthermore, there is a homomorphism $\iota : C^*(\iso(G)^\circ) \to C^*(G)$, which is injective provided $\iso(G)^\circ$ is
amenable. We can state a more general version of this result, once we distill the essential properties of $\iso(G)^\circ$. In particular, we need the following 
definitions. If $G$ is a groupoid and $\K \subseteq G$ is a subgroupoid, we say $\K$ is \emph{wide} if $\ko = \goo$. We say $\K$ is \emph{normal} if 
$\gamma^{-1}\K \gamma \subseteq \K$ for all $\gamma \in G$. A groupoid $G$ has \emph{weak containment} if $C^*(G) = C_r^*(G)$.

We now state our version of \cite[Theorem 3.1(a)]{BNRSW}. The proof from \cite{BNRSW} works almost verbatim in this case, so we omit the details.

\begin{thm}
\label{thm:inclusion}
	Let $G$ be a locally compact Hausdorff \'{e}tale groupoid, and suppose $\K \subseteq G$ is a group bundle that is open and wide in $G$.
	Then the $*$-homomorphism $\iota_0 : C_c(\K) \to C_c(G)$ defined by
	\[
		\iota_0(f)(\gamma) = \begin{cases}
			f(\gamma) & \text{ if } \gamma \in \K \\
			0 & \text{ otherwise}
		\end{cases}
	\]
	extends to a homomorphism $\iota : C^*(\K) \to C^*(G)$, which is injective provided $\K$ has weak containment (in particular, if $\K$ is amenable). If $\K$ 
	is normal, then $\iota$ descends to an injective homomorphism $\iota_r : C_r^*(\K) \to C_r^*(G)$.
\end{thm}

If $S$ is an inverse semigroup that is not cryptic, then $\G(Z)$ is an example of an open, wide, and normal subgroupoid of $\G(S)$ for which the containment $\G(Z) \subseteq \iso(G)^\circ$ may be proper. Notice that we have relaxed the requirement that $\K$ be amenable to guarantee injectivity of $\iota$. In light of Willett's \cite{willett} example of a non-amenable 
groupoid that nevertheless has weak containment, this seems like a useful distinction to make.

\begin{cor}
\label{cor:ZintoS}
	Let $S$ be an inverse semigroup with centralizer $Z$. Then there is a natural embedding $C_r^*(\G(Z)) \hookrightarrow C_r^*(\G(S))$.
\end{cor}
\begin{proof}
	Clearly $\G(Z)$ is a wide subgroupoid of $\G(S)$, and we have already shown that it is open. Therefore, we just need to check that $\G(Z)$ is a normal
	subgroupoid of $\G(S)$. This follows from the fact that $Z$ is normal in $S$: if $z \in Z$, $s \in S$, and $F \in \fE$, we have
	\[
		[s, F]^{-1} [z, \beta_s(F)] [s, F] = [s^*zs, F].
	\]
	Since $s^*zs \in Z$, $[s^*zs,F] \in \G(Z)$. The result then follows immediately from Theorem \ref{thm:inclusion}.
\end{proof}

\begin{rem}
\label{rem:ZintoS}
Corollary \ref{cor:ZintoS} yields an embedding at the level of inverse semigroups as well. In \cite[Theorem 4.4.2]{PatersonBook}, Paterson proves that there is 
an isomorphism $\psi_r^S : C_r^*(S) \to C_r^*(\G(S))$, which is characterized on elements of $S$ by
\[
	\psi_r^S(s) = \chi_{\Theta(s, D_s^\beta)}.
\]
Of course there is an analogous isomorphism $\psi_r^Z : C_r^*(Z) \to C_r^*(\G(Z))$. If we let $\iota_r$ denote the embedding of $C_r^*(\G(Z))$ into $C_r^*(\G(S))$,
then the map $(\psi_r^S)^{-1} \circ \iota_r \circ \psi_r^Z : C_r^*(Z) \to C_r^*(S)$ is injective. In light of Proposition \ref{prop:Z} and Corollary \ref{cor:ZintoS}, it is easily
checked that
\[
	(\psi_r^S)^{-1} \circ \iota_r \circ \psi_r^Z (z) = z
\]
for all $z \in Z$. In other words, the inclusion of $Z$ into $S$ extends to an embedding $j_r : C_r^*(Z) \hookrightarrow C_r^*(S)$, and the diagram
\[
	\xymatrix{ C_r^*(Z) \ar[r]^(0.45){\psi_r^Z} \ar[d]_{j_r} & C_r^*(\G(Z)) \ar[d]^{\iota_r} \\
		C_r^*(S) \ar[r]^(0.45){\psi_r^S} & C_r^*(\G(S))
	}
\]
commutes.
\end{rem}

One of the main thrusts of this paper is to determine how inverse semigroup $C^*$-algebras fit into the main results of \cite{BNRSW}. In particular, one would
hope that an analog of \cite[Theorem 3.1(c)]{BNRSW} would hold with $\iso(G)^\circ$ replaced by a group bundle (such as $\G(Z)$) that is more intrinsic
to the inverse semigroup. This seems to be a bit too much to ask for at this point. Therefore, we instead single out the inverse semigroups for which 
$\G(Z) = \iso(\G(S))^\circ$, since the results of \cite{BNRSW} will immediately apply. Thanks to Proposition \ref{prop:cryptic}, we already know that this condition
will hold when $S$ is cryptic.

\begin{thm}
	Let $S$ be a Hausdorff, cryptic inverse semigroup, and let $Z$ denote the centralizer of $E(S)$. A homomorphism $\varphi : C_r^*(S) \to A$ is injective if 
	and only if $\varphi \vert_{C_r^*(Z)}$ is injective.
\end{thm}
\begin{proof}
	In light of Remark \ref{rem:ZintoS}, it suffices to work with the universal groupoids of $S$ and $Z$. We have already shown in Proposition \ref{prop:cryptic} 
	that if $S$ is cryptic, then $\G(Z) = \iso(\G)^\circ$. It follows from Theorem 3.1(c) of \cite{BNRSW} that a homomorphism $\varphi : C_r^*(\G(S)) \to A$ is injective 
	if and only if it is injective on $C_r^*(\G(Z)) \subseteq C_r^*(\G(S))$.
\end{proof}


\section{Actions of inverse semigroups}
\label{sec:actions}

In this section, we analyze the structure of the groupoid of germs associated to an inverse semigroup action. More specifically, we are interested in the relationship
between the kernel of an action and the isotropy bundle of the corresponding groupoid. We present our uniqueness theorems at the end of this section, with an
eye toward the standard actions of an inverse semigroup $S$ on its spectra $\fE$ and $\tfE$.

Fix an action $\alpha$ of an inverse semigroup $S$ on a locally compact Hausdorff space $X$ and let $\G_{\alpha}$ denote the groupoid of germs of 
$\alpha$. The set
\[
	J = \{st^* \in S : \alpha(s) = \alpha(t)\}
\] 
is a normal inverse subsemigroup of $S$ called the \textit{kernel} of $\alpha$. In the case that $\alpha$ is the universal action of $S$ on $\fE$, then $J = Z(E)$. 
Recall that $\theta$ is defined to be the restriction of the universal action to the space of tight filters. Therefore $Z(E)$ is a subset of the kernel of $\theta$. The 
results in this section extend those of the previous section in that we relate the subgroupoid of $\G_{\alpha}$ induced by the kernel,
\[
	\G_{\alpha}(J) = \bigl\{ [s,F]: s\in J, F \in D_s \bigr\},
\]
to the isotropy bundle of $\G_\alpha$.

\begin{prop} 
\label{prop:galpha}
	Let $\alpha:S \rightarrow \mathcal{I}(X)$ be an action of $S$ on a locally compact Hausdorff space $X$ with kernel $J$. Then $\G_{\alpha}(J)$ is an open subset 
	of $\iso(\G_{\alpha})$. Moreover, if the sets $D_e^\alpha$ for $e \in E(S)$ form a base for the topology on $X$, then $\G_{\alpha}(J) = \iso(\G_{\alpha})^{\circ}$.
\end{prop}
\begin{proof}
	Fix $z = s t^*$ in $J$. Then $\alpha(z)$ is idempotent and fixes every element of $D^\alpha_{z^*z} = D^\alpha_{zz^*}$. It follows that 
	$\Theta(z,D^\alpha_{z^*z}) \subseteq \G_{\alpha}(J)$. Thus $\G_{\alpha}(J)$ is open. Now put $e = (zz^*)(z^*z)$. Then for any $x \in D^\alpha_{z^*z}$, we 
	have $x \in D^\alpha_e$. Furthermore, $z^*ze = zz^*e$. Thus $d[z,x] = r[z,x]$, so $\mathcal{G}_{\alpha}(J)$ is contained in $\iso (\G_\alpha)$.

	Now suppose that the sets $D_e$ for $e \in E(S)$ form a base for the topology on $X$. Let $[s, x] \in \iso(\G_\alpha)^{\circ}$. Then there exists 
	$e \in E(S)$ such that $[s, x] \in \Theta(s, D^\alpha_e) \subseteq \iso(\G_\alpha)$. We may assume $e \leq s^* s$, since $x \in D^\alpha_{s^*s}$ and for $e \leq f$, 
	$D^\alpha_e \subseteq D^\alpha_f$. Then $x \in D^\alpha_e$ and hence $[se, x] = [s, x]$. Notice that $\alpha_{se} = \alpha_{e}$ and hence $se \in J$. Thus 
	$[s, x] \in \G_{\alpha}(J)$.
\end{proof}

\begin{exmp} 
	Given a directed graph $\Lambda$, consider the action $\theta$ of the graph inverse semigroup $P_{\Lambda}$ on the set of tight filters of idempotents.
	The set $P_{\Lambda}$ contains $0$ and all words of the form  $xy^{*}$ where $x$ and $y$ are paths in $\Lambda$ with the same source. Multiplication is 
	defined by 
	\[
		xy^{*} \cdot uv^{*} =
			\begin{cases}
				 xzv^{*} & \mbox{if $u=yz$ for some path $z$}\\
				 x \left( vz \right)^{*} & \mbox{if $y=uz$ for some path $z$}\\
				 0 & \mbox{otherwise.}\\
			\end{cases}
	\]

	Idempotents in $P_{\Lambda}$ are of the form $xx^*$ where $x$ is a finite path in $\Lambda$, and $xx^* \leq yy^*$ if and only if $x$ extends $y$. Moreover, the semilattice $E = E(P_{\Lambda})$ is \emph{unambiguous at 0}: the product of two idempotents is nonzero only if the idempotents are comparable. Thus the filters on $E$ are linearly ordered and can be identified with paths in $\Lambda$. The tight filters are the set of all infinite paths together with the finite paths that end at singular vertices (see \cite{PatersonGraph}). 

If $\Lambda$ is row-finite and does not contain sinks, then there are no singular vertices and the tight filters can be identified with infinite paths. By \cite[Proposition 3(ii)]{PatersonGraph}, sets of the form $D^\theta_e$ for $e \in E(S)$ form a base for the topology on $\tfE$. Thus for graphs that are row-finite and do not contain sinks, $\G_{\theta}(J) = \iso(\G_{\theta})^{\circ}$. In this case, the Cartan subalgebra for which the generalized Cuntz-Krieger uniqueness theorem of \cite{BNRSW} holds is generated by the kernel $J$. In Theorem \ref{thm:kerneluniqueness} below, we give a uniqueness theorem that holds for the subalgebra of $C_r^*(\G_\alpha)$ generated by the kernel of an action satisfying the hypotheses of the above proposition.
\end{exmp}

Note that for any inverse semigroup $S$, $Z(E)$ is contained in the kernel $J$ of the tight action $\theta$. Equality holds if and only if $\theta$ is idempotent 
separating, which is equivalent to the statement that the map $e \mapsto D_e$ is injective. It turns out that we can place restrictions on the semilattice that force
$\theta$ to be idempotent separating. 

A semilattice $E$ with zero is said to be \emph{$0$-disjunctive} if for all $0 < e < f$, there exists $0 < e' < f$ such that $ee' = 0$. Recall that Exel \cite{ExelBig} showed 
that the ultrafilters form a dense subset of the tight filters. Given a nonzero idempotent $e$ of $S$, let $U_e$ be the set of ultrafilters containing $e$. Then 
$D_e^\theta = \overline{U}_e$ for each $e$. Lawson showed that for $0$-disjunctive inverse semigroups, $e \mapsto U_e$ is injective, and that the inverse 
semigroup of a graph is $0$-disjunctive if and only if no vertex has in-degree 1 \cite[Lemma 2.8, Lemma 2.15]{LawsonCompactable}. 

\begin{lem} 
	Let S be an inverse semigroup with semilattice of idempotents $E$. If $E$ is $0$-disjunctive, then the homomorphism $\theta: S \to \mathcal{I}(\tfE)$ is 
	idempotent separating.
\end{lem}
\begin{proof}  
	Since $E$ is $0$-disjunctive, the map $e \mapsto U_e$ is injective. Suppose for idempotents $e$ and $f$ that we have $\theta(e) = \theta(f)$. Then 
	$D_e^\theta = D_f^\theta$. Let $F \in U_e$. Then $F \in D_e^\theta = D_f^\theta$ and so $F$ contains $f$. Thus $U_e \subseteq U_f$. By symmetry, $U_e = U_f$. 
	Since $E$ is $0$-disjunctive, $e = f$.
\end{proof} 
 
In light of the previous lemma, we focus on the tight action $\theta$ of an inverse semigroup $S$ having a $0$-disjunctive semilattice $E$. Recall that the groupoid of
germs $\G_\theta$ is the tight groupoid $\Gt(S)$; we will use these two notations interchangeably.

\begin{prop}
\label{prop:stabilizers} 
	Let S be an inverse semigroup with semilattice of idempotents $E$, and suppose that $E$ is $0$-disjunctive. Let $\theta$ be the homomorphism 
	$\theta: S \to \mathcal{I}(\tfE)$. Then $\G_{\theta}(Z) = \iso(\G_\theta)^\circ$ 
\end{prop}
\begin{proof} 
	We will show that $\G_{\theta}(Z)_F = \iso(\G_\theta)_{F}^{\circ}$ for every tight filter $F \in \fE$. The containment 
	$\G_{\theta}(Z)_F \subseteq \iso(\G_\theta)_{F}^{\circ}$ holds for any filter $F$. The proof is similar to the proof of Proposition \ref{prop:Z}. For the other 
	containment, we first assume that $F$ is an ultrafilter. Suppose that $[s,F] \in \iso(\G_\theta)_{F}^{\circ}$. Then there exist idempotents $x, x_1, \dots, x_n$
	such that 
	\[ 
		[s,F] \in \Theta(s, N^{x}_{x_1,x_2, \dots, x_n}) \subseteq \iso(\G_\theta)^{\circ} 
	\]
	Since $E$ is $0$-disjunctive, the sets $U_e$ form a base for the space of ultrafilters (see \cite[Proposition 5.10]{SteinbergSimple}). Therefore, there is an 
	idempotent $e \in F$ such that $U_e$ is contained in the set of ultrafilters in $N^{x}_{x_1,x_2, \dots, x_n}$. Then 
	$\overline{U}_e = D_{e}^\theta \subseteq N^{x}_{x_1,x_2, \dots, x_n}$.  Thus $[s,F] \in \Theta(s,D_e^\theta) \subseteq \iso(\G_\theta)$. We may assume that 
	$e \leq s^*s$ (possibly by replacing $e$ with $e s^*s$).

	Now $\theta(se)$ is a partial bijection with domain $D_{ (se)^* se}^\theta = D_{e}^\theta$. Also, since $\Theta(s,D_{e}^\theta) \subseteq \iso(\G_\theta)$, 
	$\theta(se)$ fixes every element of $D_{e}^\theta$. Thus $\theta(se) = \theta(e)$ and $se \in J_{\theta}$, where $J_{\theta}$ is the kernel of $\theta$. Since $\theta$ is idempotent separating, $se \in Z$. Finally we have $[s,F] = [se,F] \in \G_{\theta}(Z)_F$.

	We have verified that $\G_{\theta}(Z)_F = \iso(\G_\theta)_{F}^{\circ}$ for any ultrafilter $F$. As in the proof of Proposition \ref{prop:cryptic}, $\G_{\theta}(Z)$ 
	and $\iso(\G_\theta)^{\circ}$ are both open in $\G_\theta$ and thus have continuously varying stabilizers. Since the set of ultrafilters is dense in $\tfE$ it follows that 
	$\G_{\theta}(Z)_F = \iso(\G_\theta)_{F}^{\circ}$ for each tight filter $F$.
\end{proof}

\begin{cor} 
	If $E$ is $0$-disjunctive and fundamental then $\Gt(S)$ is essentially principal.
\end{cor}

Notice that the above corollary implies \cite[Corollary 5.11]{SteinbergSimple}. Moreover, we have managed to remove the hypothesis that $\Gt(S)$ is Hausdorff.

Under certain conditions, the results of this section yield a uniqueness theorem for the groupoid of germs of an inverse semigroup action. Specifically, we obtain a uniqueness theorem for the tight groupoid of an inverse semigroup.

\begin{thm}\label{thm:kerneluniqueness} 
	Let $S$ be a Hausdorff inverse semigroup, and let $\alpha:S \rightarrow \mathcal{I}(X)$ be an action of $S$ on a locally compact Hausdorff space $X$ with kernel $J$. Suppose the collection
	\[
		\{D_e^\alpha : e \in E(S)\}
	\]
	forms a base for the topology on $X$, and each $D_e^\alpha$ is clopen in $X$. Then a homomorphism $\varphi : C_r^*(\G_\alpha) \to A$ is injective if and only 
	if $\varphi \vert_{C_r^*(\G_\alpha(J))}$ is injective.
\end{thm}
\begin{proof}
	The assumption that $S$ is Hausdorff and the sets $D_e^\alpha$ are clopen implies that $\G_\alpha$ is Hausdorff by \cite[Proposition 2.3]{SteinbergSimple}. 
	By Proposition \ref{prop:galpha}, we have 
	$\G_\alpha(J) = \iso(\G_\alpha)^\circ$, so the result follows from Theorem 3.1 of \cite{BNRSW} with $H = \G_\alpha(J)$.
\end{proof}

\begin{thm}
\label{thm:tightuniqueness}
	Let S be a Hausdorff inverse semigroup with semilattice of idempotents $E$ and centralizer $Z$, and suppose $E$ is $0$-disjunctive. Let 
	$\theta: S \to \mathcal{I}(\tfE)$
	denote the tight action of $S$, and let $\G_\theta$ and $\G_\theta(Z)$ denote the tight groupoids of $S$ and $Z$, respectively. Then a homomorphism
	$\varphi : C_r^*(\G_\theta) \to A$ is injective if and only if $\varphi \vert_{C_r^*(\G_\theta(Z))}$ is injective.
\end{thm}
\begin{proof}
	By Proposition \ref{prop:stabilizers}, we have $\G_\theta(Z) = \iso(\G_\theta)^\circ$ since $E$ is $0$-disjunctive. Therefore, we can just apply \cite[Theorem 3.1]{BNRSW}.
\end{proof}


\section{Amenability of the Universal Groupoid}
\label{sec:amenability}
In this section, we establish conditions that guarantee the amenability of the universal groupoid for a certain class of inverse semigroups. We plan to invoke 
\cite[Corollary 4.5]{ren-wil}, which requires the existence of a certain cocycle on $\G(S)$. (Actually, the original result \cite[Theorem 9.3]{spielberg} is sufficient for 
our purposes.) These conditions will also guarantee that $\G(Z)$ is closed, which in turn allows us to build a conditional expectation from $C_r^*(\G(S))$ to 
$C_r^*(\G(Z))$ in certain cases.

Let $S$ be an inverse semigroup, and recall that $S/\mu$ denotes the Munn quotient of $S$. We aim to relate the amenability of $\G(S)$ to that of $\G(Z)$ and 
$\G(S/\mu)$. To do so, we'll use \cite[Theorem 5.3.14]{ananth-renault}, which is the basis for the aforementioned cocycle results from \cite{ren-wil} and \cite{spielberg}. 
The first step is to produce a \emph{strongly surjective} homomorphism $\varphi : \G(S) \to \G(S/\mu)$, meaning that $\varphi(\G(S)_F) = \G(S/\mu)_F$ for all $F \in \fE$.

\begin{prop}
\label{prop:51}
	Define $\varphi : \G(S) \to \G(S/\mu)$ by
	\[
		\varphi([s, F]) = [\mu(s), F].
	\]
	Then $\varphi$ is a well-defined, continuous, and strongly surjective groupoid homomorphism.
\end{prop}
\begin{proof}
	First observe that if $[s, F] = [t, F]$ in $\G(S)$, then there is an $e \in F$ such that $se = te$. But then $\mu(s)e = \mu(se) = \mu(te) = \mu(t)e$, so 
	$[\mu(s), F] = [\mu(t), F]$, and $\varphi$ is well-defined. It is straightforward to see that $\varphi$ is a homomorphism: it is easily verified that 
	$\theta_{\mu(s)}(F) = \theta_s(F)$, so
	\begin{align*}
		\varphi([t, \theta_s(F)][s, F]) &= \varphi([ts,F]) \\
			&= [\mu(ts),F] \\
			&= [\mu(t), \theta_{\mu(s)}(F) ][\mu(s), F] \\
			&= \varphi([t, \theta_s(F)]) \varphi([s, F]),
	\end{align*}
	Now we check that $\varphi$ is continuous. Suppose $[s_i, F_i] \to [s, F]$ in $\G(S)$. Let $U \subseteq \fE$ be an open neighborhood of $F$, and let 
	$\Theta(s, U)$ and $\Theta(\mu(s), U)$ be the corresponding basic open sets in $\G(S)$ and $\G(S/\mu)$, respectively. Since $[s_i, F_i]$ converges to 
	$[s, F]$, eventually $[s_i, F_i] \in \Theta(s, U)$. Thus $[s_i, F_i] = [s, F_i]$ eventually, so there are idempotents $e_i \in F_i$ such that $s_ie_i = se_i$. But 
	then
	\[
		\mu(s_i) e_i = \mu(s_i e_i) = \mu(se_i) = \mu(s) e_i,
	\]
	so $[\mu(s_i), F_i] = [\mu(s), F_i]$ in $\G(S/\mu)$. Therefore, $[\mu(s_i), F_i]$ is eventually in $\Theta(\mu(s), U)$, so $[\mu(s_i), F_i] \to [\mu(s), F]$. It follows 
	that $\varphi$ is continuous.
	
	Finally, it is straightforward to see that $\varphi$ is strongly surjective: simply notice that for any $F \in \fE$,
	\[
		\G(S/\mu)_F = \{ [\mu(s), F] : s \in S \} = \varphi \bigl( \{[s, F] : s \in S\} \bigr) = \varphi(\G(S)_F). \qedhere
	\]
\end{proof}

In order to apply \cite[Theorem 5.3.14]{ananth-renault}, we need to understand the kernel of the homomorphism $\varphi$. It is clear that 
$\G(Z) \subseteq \ker(\varphi)$. Observe that if $\varphi([s, F]) = [e, F]$ for some idempotent $e \in F$, then $[\mu(s), F] = [e, F]$, so $\mu(s)e = e$. In other 
words, $se$ is $\mu$-related to $e$ in $S$, so $se \in Z$. Thus $s \in \up Z$. This set could be quite hard to understand in general. However, we can do better 
if we assume that $S/\mu$ is $0$-$E$-unitary. In this case, $\mu(s)e=e$ implies that $\mu(s)$ is an idempotent. That is, $\mu(s) = f$ for some $f \in E(S)$, so 
$s$ is $\mu$-related to $f$. This forces $s \in Z$, so $\ker(\varphi) = \G(Z)$. Thus we have shown:

\begin{prop}
	Let $S$ be an inverse semigroup, and assume the Munn quotient $S/\mu$ is 0-$E$-unitary. Then the kernel of the homomorphism $\varphi$ from Proposition
	\ref{prop:51} is $\G(Z)$.
\end{prop}

It now follows from \cite[Theorem 5.3.14]{ananth-renault} that $\G(S)$ is amenable if and only if $\G(Z)$ and $\G(S/\mu)$ are. We can refine this result even 
further---the question of whether $\G(Z)$ is amenable can be traced back to the structure of $S$ itself.

\begin{lem}
	Let $S$ be a Clifford semigroup, and for each $e \in E(S)$, let $H_e = \{s \in S : s^*s = ss^* =e\}$ denote the maximal subgroup of $S$ at $e$. Then the 
	groupoid $\G(S)$ is amenable if and only if $H_e$ is an amenable group for each $e \in E(S)$.
\end{lem}
\begin{proof}
	We proved earlier that for any filter $F \in \fE$, the fiber of $\G(S)$ at $F$ is the direct limit
	\[
		\G(S)_F = \varinjlim_{e \in F} H_e.
	\]
	Since direct limits of amenable groups are amenable, it follows that the fibers of $\G(S)$ are all amenable provided every $H_e$ is amenable. This guarantees 
	that $\G(S)$ is amenable by \cite[Theorem 4]{sims-williams2013}. On the other hand, if $F=\up e$ is a principal filter, then $\G(S)_F = H_e$. Therefore, each 
	$H_e$ is a subgroup of $\G(S)$. It follows that if $\G(S)$ is amenable, then each $H_e$ must be amenable.
\end{proof}

\begin{thm}
\label{thm:amenable}
	Let $S$ be an inverse semigroup such that $S/\mu$ is $0$-$E$-unitary. Then the groupoid $\G(S)$ is amenable if and only if $\G(S/\mu)$ is amenable and 
	each $\mu$-class
	\[
		Z_e = \left\{ s \in S : \mu(s) = \mu(e) \right\}
	\]
	is an amenable group.
\end{thm}
\begin{proof}
	Since $\varphi : \G(S) \to \G(S/\mu)$ is strongly surjective, $\G(S)$ is amenable if and only if $\ker \varphi$ and $\G(S/\mu)$ are by Theorem 5.3.14 of
	\cite{ananth-renault}. But $S/\mu$ is $0$-$E$-unitary, so $\ker \varphi = \G(Z)$. But we know from the previous lemma that this group bundle is amenable if 
	and only if every $Z_e$ is amenable.
\end{proof}

While it is easy to characterize when $\G(Z)$ is amenable, it may be difficult to do the same for $\G(S/\mu)$ in general. However, we can say something definitive
if we assume that $S$ does not have a 0 element. In this case, we assume $S/\mu$ is $E$-unitary, and we let $G_\mu = \sigma(S/\mu)$ denote the maximal group 
image of $S/\mu$. Then we still have a homomorphism $\varphi : \G(S) \to \G(S/\mu)$ with $\ker \varphi = \G(Z)$, so $\G(S)$ is amenable if and only if $\G(Z)$ and 
$\G(S/\mu)$ are. However, we now have a way of relating the amenability of $\G(S/\mu)$ to that of $G_\mu$.

\begin{prop}
	Let $S$ be an inverse semigroup with maximal group image $\sigma(S)$. The map $c : \G(S) \to \sigma(S)$ defined by
	\[
		c([s, F]) = \sigma(s)
	\]
	is a continuous cocycle on $\G(S)$. If $S$ is $E$-unitary, then $\ker c = \G(S)^{\scriptscriptstyle{(0)}}$.
\end{prop}
\begin{proof}
	We first need to check that $c$ is well-defined. Suppose $[s, F] = [t, F]$ in $\G(S)$. Then there is an idempotent $e \in F$ such that $se = te$, so
	\begin{align*}
		c([s, F]) = \sigma(s) = \sigma(se) = \sigma(te) = \sigma(t) = c([t, F]).
	\end{align*}
	Next, we check that $c$ is a homomorphism. Well,
	\begin{align*}
		c([t, \theta_s(F)][s, F]) &= c([ts, F]) \\
			&= \sigma(ts) \\
			&= \sigma(t) \sigma(s) \\
			&= c([t,\theta_s(F)]) c([s, F]).
	\end{align*}
	Now suppose $[s_i, F_i] \to [s, F]$ in $\G(S)$. Then $F_i \to F$ and for sufficiently large $i$, there are idempotents $e_i \in F_i$ such that $s_i e_i = se_i$. 
	But then
	\[
		c([s_i, F_i]) = \sigma(s_i) = \sigma(s_ie_i) = \sigma(se_i) = \sigma(s),
	\]
	so $c([s_i, F_i])$ is eventually constant, hence convergent. Thus $c$ is continuous.
	Finally, if we assume $S$ is $E$-unitary, then
	\begin{align*}
		\ker c &= \left\{ [s, F] \in \G(S) : \sigma(s) = 1 \right\} \\
			&= \left\{ [s, F] \in \G(S) : se = e \text{ for some } e \in E(S) \right\} \\
			&= \left\{ [s, F] \in \G(S) : s \in E(S) \right\} \\
			&= \G(S)^{\scriptscriptstyle{(0)}}. \qedhere
	\end{align*}

\end{proof}

\begin{prop}
	If $S$ is an $E$-unitary inverse semigroup, then $\G(S)$ is amenable if and only if $\sigma(S)$ is an amenable group.
\end{prop}
\begin{proof}
	From the previous proposition, we have a cocycle $c : \G(S) \to \sigma(S)$ with $\ker c = \G(S)^{\scriptscriptstyle{(0)}}$. Thus we know from 
	\cite[Corollary 4.5]{ren-wil} that if $\G(S)^{\scriptscriptstyle{(0)}}$ and $\sigma(S)$ are amenable, then so is $\G(S)$. But the former is always amenable, 
	so if $\sigma(S)$ is amenable, then so is $\G(S)$. On the other hand, supppose that $\G(S)$ is amenable. Since $S$ is $E$-unitary, we may assume that $E$ does not contain a zero. (Otherwise $\sigma(S)$ is trivial and hence amenable.) Thus the semilattice $E$ is a filter. Notice that for all $s \in S$, $[s^*s, E] = [ss^*, E]$ since $s^*s f = ss^* f$ where $f = (s^*s)(ss^*)$. Thus the group $\G(S)^{E}_{E}$ consists of equivalence classes $[s,E]$ for all $s \in S$ where $[s,E] = [t,E]$ if and only if $\sigma(s) = \sigma(t)$. It is then easy to check that the map $[s,E] \mapsto \sigma(s)$ is an isomorphism from $\G(S)^{E}_{E}$ to $\sigma(S)$. Since $\G(S)$ amenable implies that $\G(S)^{E}_{E}$ is amenable, we have that $\sigma(S)$ is amenable.
\end{proof}

By applying this proposition to $S/\mu$, we immediately obtain the following corollary. Note that this result strengthens \cite[Corollary 2.6]{milan}.

\begin{cor}
	Let $S$ be an inverse semigroup for which $S/\mu$ is $E$-unitary. The groupoid $\G(S)$ is amenable if and only if each $\mu$-class
	\[
		Z_e = \left\{ s \in S : \mu(s) = \mu(e) \right\}
	\]
	is amenable and the maximal group image $G_\mu = \sigma(S/\mu)$ is amenable.
\end{cor}

An inverse semigroup whose semilattice of idempotents is isomorphic to the natural numbers under the reverse of the usual ordering is called an 
{\it inverse $\omega$-semigroup}. The above result allows for a characterization of amenability for the universal groupoid of such a semigroup.

\begin{cor}
	Let $S$ be an inverse $\omega$-semigroup. Then $\G(S)$ is amenable if and only if each maximum subgroup
	\[
		H_e = \{  s \in S : ss^* = s^*s=e \}
	\]
	of $S$ is amenable.
\end{cor}
\begin{proof} If $S$ is an inverse $\omega$-semigroup then $S$ is cryptic and $S/\mu$ is a full subsemigroup of the bicyclic monoid (see 
\cite[Section 5.4]{LawsonBook}). Thus $H_e = Z_e$ for all $e$ in $E(S)$ and $\sigma(S/\mu) \subseteq \mathbb{Z}$.
\end{proof}

We would like to obtain a more general ``two out of three'' amenability theorem that doesn't require assumptions about the Munn quotient $S/\mu$. We have to pay for this by imposing more hypotheses on $S$ itself. In particular, we need to assume that the extension $Z \to S \to S/\mu$ is \emph{split}, so there is a transversal $r : S/\mu \to S$. Given $s \in S$, we will use $s_r$ to denote $r(\mu(s))$. 

In the case that we have a split extension, we can appeal to results about semidirect products of groupoids. Let $G$ and $\K$ be groupoids, and suppose there is a map $p : \K \to \goo$ satisfying $p = p \circ r = p \circ d$, and that $G$ acts on the left of $\K$. The \emph{semidirect product} of $G$ and $\K$ is the groupoid
\[
	\K \rtimes G = \bigl\{ (\eta, \gamma) \in \K \times G : p(\eta) = r(\gamma) \bigr\},
\]
with
\[
	(\eta_1, \gamma_1)(\eta_2, \gamma_2) = \bigl( \eta_1 (\gamma_1 \cdot \eta_2) , \gamma_1 \gamma_2 \bigr)
\]
and
\[
	(\eta, \gamma)^{-1} = (\gamma^{-1} \cdot \eta^{-1}, \gamma^{-1})
\]
whenever the product makes sense. The unit space of $\K \rtimes G$ can be identified with $\ko$, and the range and source maps are given by
\[
	d(\eta, \gamma) = \gamma^{-1} \cdot d(\eta), \quad r(\eta, \gamma) = r(\eta).
\]

\begin{exmp}
\label{exmp:semidirect}
	Let $\Gamma$ be a groupoid, $\K \subseteq \Gamma$ a group bundle, and $G \subseteq \Gamma$ a wide subgroupoid that normalizes $\K$. Then 
	$\ko = \goo = \Gamma^{\scriptscriptstyle{(0)}}$, and the structure map $p : \K \to \goo$ is just the bundle map of $\K$. Note that $G$ acts on the left of $\K$ via 
	conjugation:
	\[
		\gamma \cdot \eta = \gamma \eta \gamma^{-1}.
	\]
	Then we have
	\[
		\K \rtimes G = \bigl\{ (\eta, \gamma) \in \K \times G : p(\eta) = r(\gamma) \bigr\},
	\]
	with the groupoid operations given by
	\[
		(\eta_1, \gamma_1)(\eta_2, \gamma_2) = \bigl( \eta_1(\gamma_1 \eta_2 \gamma_1^{-1}), \gamma_1 \gamma_2)
	\]
	and
	\[
		(\eta, \gamma)^{-1} = (\gamma^{-1} \eta^{-1} \gamma, \gamma^{-1}).
	\]
\end{exmp}

The semidirect product that we will consider is essentially that of Example \ref{exmp:semidirect}, with $\Gamma = \G(S)$, $\K = \G(Z)$, and $G = \G(S/\mu)$.

\begin{prop}
	Let $S$ be an inverse semigroup, $Z$ the centralizer of the idempotents, and $S/\mu$ the Munn quotient, and suppose the extension
	\[
		Z \to S \to S/\mu
	\] 
	is split. Then the map $\varphi : \G(Z) \rtimes \G(S/\mu) \to \G(S)$ defined by
	\[
		\varphi \bigl( [z, \theta_{s_r}(F)], [s_r, F] \bigr) = [zs_r, F]
	\]
	is an isomorphism of topological groupoids.
\end{prop}
\begin{proof}
	First note that we can identify $\G(S/\mu)$ with a subgroupoid of $\G(S)$, and that $\G(S/\mu)$ acts on $\G(Z)$ by conjugation (since $S/\mu$ acts on $Z$
	via conjugation). Thus it makes sense to build the semidirect product $\G(Z) \rtimes \G(S/\mu)$. Observe that $\varphi$ is well-defined, continuous, and open, 
	since it is simply given by multiplication in $\G(S)$. Thus it suffices to check that $\varphi$ is an algebraic isomorphism.
	
	Let $([z, \theta_{s_r}(F)],[s_r, F]), ([w, \theta_{t_r}(F')], [t_r, F']) \in \G(Z) \rtimes \G(S/\mu)$. Then observe that
	\begin{align*}
		\varphi \bigl( ([z, \theta_{s_r}(F)],[s_r, F]), ([w, \theta_{t_r}(F')], [t_r, F']) \bigr) &= \varphi \bigl( [zs_rws_r^*, \theta_{s_r}(F)], [s_r t_r, F'] \bigr) \\
			&= [zs_rws_r^*s_rt_r, F'] \\
			&= [zs_rs_r^*s_rwt_r, F'] \\
			&= [z s_r wt_r, F']
	\end{align*}
	since $w \in Z$. On the other hand,
	\begin{align*}
		\varphi \bigl( [z, \theta_{s_r}(F)],[s_r, F] \bigr) \varphi \bigl( [w, \theta_{t_r}(F')], [t_r, F'] \bigr) &= [zs_r, F][w t_r, F'] \\
			&= [zs_r w t_r, F'],
	\end{align*}
	so $\varphi$ is a homomorphism.
	To see that $\varphi$ is surjective, put $z = ss_r^*$. Then $z \in Z$ and $s = zs_r$. Moreover, $s^*s = s_r^*z^*zs_r \in F$, so $z^*z \in \theta_{s_r}(F)$ and 
	$s_r^*s_r \in F$. Thus
	\[
		[s, F] = [z, \theta_{s_r}(F)][s_r, F] = \varphi([z, \theta_{s_r}(F)],[s_r, F])
	\]
	and $\varphi$ is surjective. For injectivity, suppose that 
	\[
		\varphi([z, \theta_{s_r}(F)], [s_r, F]) = \varphi([w, \theta_{t_r}(F')],[t_r, F']),
	\]
	i.e., $[zs_r, F] = [wt_r,F']$. Then $F = F'$, and there is an idempotent $e \in F$ such that $zs_re = wt_r e$. We claim first that $[s_r, F] = [t_r, F]$. We know that
	$se = te$, and $r \circ \mu$ is a homomorphism, so $s_r e = t_r e$. Thus $[s_r, F] = [t_r, F]$, and it follows by cancellation that $[z, \theta_{s_r}(F)] = 
	[w, \theta_{t_r}(F)]$. Thus $\varphi$ is an isomorphism.
\end{proof}

\begin{thm}
\label{thm:splitamen}
	Let $S$ be an inverse semigroup, and suppose $S$ is a split extension of $Z$ by $S/\mu$. Then $\G(S)$ is amenable if and only if $\G(S/\mu)$ is amenable
	and $Z_e$ is amenable for all $e \in E(S)$.
\end{thm}
\begin{proof}
	By the previous proposition, we have $\G(S) \cong \G(Z) \rtimes \G(S/\mu)$. As described in the discussion following Definition 5.3.26 in \cite{ananth-renault},
	there is a homomorphism $c : \G(S) \cong \G(Z) \rtimes \G(S/\mu) \to \G(S/\mu)$, which is strongly surjective with $\ker c = \G(Z)$. It then follows from
	\cite[Theorem 5.3.14]{ananth-renault} that $\G(S)$ is amenable if and only if $\G(S/\mu)$ and $\G(Z)$ are. As we have already seen, the latter is amenable if and 
	only if every group $Z_e$ is amenable.
\end{proof}

We end with a few comments on conditional expectations. Thanks to the following analog of \cite[Proposition 4.1]{BNRSW}, we can build a conditional expectation 
from $\G(S)$ to $\G(Z)$ whenever the latter is closed. The proof is identical to the one found in \cite{BNRSW}.

\begin{thm}
\label{thm:condexp}
	Let $G$ be a locally compact Hausdorff \'{e}tale groupoid, and let $H \subseteq G$ be a group bundle that is open, wide, and normal. If $H$ is closed, then there 
	is a faithful conditional expectation $\Phi_r : C_r^*(G) \to C_r^*(H)$ characterized by $\Phi_r(f) = f \vert_H$ for all $f \in C_c(G)$. If $H$ has weak containment, 
	then there is a conditional expectation $\Phi : C^*(G) \to C^*(H)$, which is not necessarily faithful.
\end{thm}
\begin{proof}
	The proof of the first assertion is identical to the one found in \cite{BNRSW}. The second assertion is also proven in the same way as in \cite{BNRSW}, once
	one realizes that it not necessary to assume that $H$ is amenable. The crucial hypothesis is that $C^*(H) = C_r^*(H)$, i.e., that $H$ has weak containment.
\end{proof}

We know that $\G(Z)$ is always an open, wide, normal subgroupoid of $\G(S)$. If $S$ is fundamental, then $\G(Z) = \G(S)^{\scriptscriptstyle{(0)}}$ and hence it is closed. 
It is well known that there is a conditional expectation $\Phi_r : C_r^*(\G) \to C_r^*(\G)^{\scriptscriptstyle{(0)}}$. The results of this section give us three other situations 
where $\G(Z)$ is guaranteed to be closed.

\begin{thm}
	Let $S$ be a Hausdorff inverse semigroup satisfying any of the following conditions:
	\begin{enumerate}
		\item $S$ is fundamental.
		\item $S/\mu$ $E$-unitary. 
		\item $S/\mu$ is $0$-$E$-unitary.
		\item The extension $Z \to S \to S/\mu$ is split.
	\end{enumerate}
	Then there is a faithful conditional expectation $\Phi_r : C_r^*(\G(S)) \to C_r^*(\G(Z))$, which is characterized
	on $C_c(\G(S))$ by restriction of functions:
	\[
		\Phi_r(f) = f \vert_{\G(Z)}.
	\]
	If $\G(Z)$ has weak containment, then there is a (not necessarily faithful) conditional expectation $\Phi : C^*(\G(S)) \to C^*(\G(Z))$.
\end{thm}
\begin{proof}
	We have already discussed the first case. In the second and third cases, we know that $\G(Z)$ is the kernel of a continuous cocycle. Moreover, $\G(S/\mu)^{(0)}$ 
	is closed since $\G(S/\mu)$ is Hausdorff, so $\G(Z)$ must be closed. If the extension $Z \to S \to S/\mu$ is split, then the proof of Theorem \ref{thm:splitamen} 
	shows that $\G(Z)$ is again the kernel of a homomorphism $c: \G(S) \to \G(S/\mu)$. If $\G(S)$ is Hausdorff, then so is $\G(S/\mu)$, which implies that
	$\ker c = \G(Z)$ is closed.
\end{proof}

Finally, we include an example of an inverse semigroup $S$ for which $\iso(\G(S))^{\circ}$ fails to be closed in the universal groupoid of $S$, 
but the groupoid of $Z(E)$ is closed. Let $E = \{x_i,L_i,R_i : i > 0\} \cup \{0,1\}$ be the semilattice pictured below.

\begin{center}\begin{tikzpicture}
	\node (1) at (2,5) {$1$};
	\node[draw=none] (eL) at (2,4.25) {};
	\node[draw=none] (eR) at (3,4.25) {};
	\node[draw=none] (ell1) at (2,4) {$\vdots$};
	\node[draw=none] (L3) at (1,3.75) {$L_3$};
	\node[draw=none] (R3) at (3,3.75) {$R_3$};
	\node (x2) at (2,3) {$x_2$};
	\node (L2) at (1,2.25) {$L_2$};
	\node (R2) at (3,2.25) {$R_2$};
	\node (x1) at (2,1.5) {$x_1$};
	\node (L1) at (1,.75) {$L_1$};
	\node (R1) at (3,.75) {$R_1$};
	\node (0) at (2,0) {$0$};
	\draw (0)--(L1)--(x1)--(L2)--(x2)--(R2)--(x1)--(R1)--(0) (x2)--(L3) (x2)--(R3) (eL)--(1) ;
	
\end{tikzpicture}
\end{center}

Let $S$ be the inverse subsemigroup of $T_E$ generated by $E$ and the automorphism $\alpha$ of $\down 1 = E$ that interchanges $L_i$ and $R_i$ for all $i > 0$ 
and leaves all other idempotents fixed. Then $S$ is a fundamental inverse semigroup, since it is a full subsemigroup of $T_E$. Thus $\G(Z)$ is principal (and equal to
$\G(S)^{\scriptscriptstyle{(0)}}$, in fact), hence closed in the universal groupoid of $S$. Note that $\Theta(\alpha, N^{x_i}_{L_{i},R_{i}}) = \{[\alpha, \up x_i]\} \subseteq 
\iso(\G(S))^{\circ}$.

Since $\up 1$ is the singleton $\{1\}$, the basic open sets containing $[\alpha, \up 1]$ are all of the form $\Theta(\alpha, N^{1}_{e_1, e_2, \dots, e_n})$, where $e_j < 1$ for $1 \leq j \leq n$. There exists $N$ such that $e_j < x_N$ for $1 \leq j \leq n$. So for $i > N$, $[\alpha, \up x_i] \in \Theta(\alpha,N^{1}_{e_1, e_2, \dots, e_n})$. Thus $[\alpha, \up x_i] \longrightarrow [\alpha, \up 1]$. Similarly $[\alpha, \up L_{i+1}] \in \Theta(\alpha,N^{1}_{e_1, e_2, \dots, e_n})$ for all $i>N$ and hence $[\alpha, \up L_{i+1}] \longrightarrow [\alpha, \up 1]$ . Since $[\alpha, \up L_{i+1}] \not\in \iso(\G(S))$, we see that 
$[\alpha, \up 1] \not\in \iso(\G(S))^{\circ}$ and hence $\iso(\G(S))^{\circ}$ is not closed.

\bibliographystyle{amsplain}
\bibliography{SemigroupBib.bib}
\end{document}